\documentclass{amsart}
\usepackage{amsmath,amssymb,amsthm}
\usepackage{young}
\usepackage{booktabs}
\usepackage[usenames,dvipsnames]{color}
\usepackage{comment}
\usepackage{graphicx}
\usepackage{latexsym}
\usepackage[latin1]{inputenc}
\usepackage{mathtools}
\usepackage[shortlabels]{enumitem}
\usepackage{thmtools}
\usepackage{url}
\usepackage[all,cmtip]{xy}
\usepackage{tikz}
\usepackage[pagebackref]{hyperref} 
\usepackage{cite}
\usepackage{hyperref}

\newcommand{\arxiv}[1]{\url{http://arxiv.org/abs/#1}}

\newcommand{\F}{\mathcal{F}}

\newtheorem{thm}{Theorem}[section]
\newtheorem{lem}[thm]{Lemma}

\newtheorem{cor}[thm]{Corollary}
\theoremstyle{definition}

\theoremstyle{definition}
\newtheorem{definition}[equation]{Definition}
 
\newcommand\mat[2]{{\mathrm{Mat}}_{#1\times #2}}
\newcommand\umat[2]{{\mathrm{Mat}}_{#1\times #2}^*}
\newcommand\Cat{\operatorname{Cat}}
\newcommand\SYT{\operatorname{SYT}}
\newcommand\CT{\operatorname{CT}}
\newcommand\vol{\operatorname{vol}}
\newcommand{\CTn}[1]{\CT_{x_{#1}}\cdots\CT_{x_1}}

\newtheorem*{lemma*}{Lemma}

\newtheorem*{corollary*}{Corollary}
\newtheorem*{theorem*}{Theorem}
\newtheorem*{theorem1*}{Theorem \ref{thm:2}}
\newtheorem*{theorem2*}{Theorem \ref{thm:abm}}
\newtheorem*{theorem3*}{Theorem \ref{111}}

\def\ee{{\bf e}}
  
   \def\vol{{\rm vol}}
 
\def\vv{{\rm v}}

 \def\f_H{{\bf w}}
 \def\f{{\bf f}}
 \def\a{{\bf a}}
\def\R{\mathbb{R}}

\def\Z{\mathbb{Z}}
\def\g{{\bf a}}

 \def\F{\mathcal{F}}

\def\ee{{\bf e}}
  
   \def\vol{{\rm vol}}
 
\def\vv{{\rm v}}

\def\aa{{\bf a}}
 \def\f_H{{\bf w}}
 \def\f{{\bf f}}
 \def\a{{\bf a}}
\def\R{\mathbb{R}}

\def\Z{\mathbb{Z}}
\def\g{{\bf a}}

 \def\F{\mathcal{F}}

\def\CT{CT}
\def\ee{{\bf e}}
  
   \def\vol{{\rm vol}}
 
\def\vv{{\rm v}}

\def\aa{{\bf a}}
 \def\f_H{{\bf w}}
 \def\f{{\bf f}}
 \def\a{{\bf a}}
\def\vvv{{\bf v}}
\def\R{\mathbb{R}}

\def\Z{\mathbb{Z}}
\def\g{{\bf a}}

 \def\F{\mathcal{F}}

\newcommand{\be}{\begin{equation}}

\newcommand{\eee}{\end{equation}}

\newcommand{\bd}{\begin{definition}}
\newcommand{\ed}{\end{definition}}
\newcommand{\bt}{\begin{thm}}
\newcommand{\et}{\end{thm}}
\newcommand{\bl}{\begin{lem}}
\newcommand{\el}{\end{lem}}
\newcommand{\bp}{\begin{proposition}}
\newcommand{\ep}{\end{proposition}}
\newcommand{\bc}{\begin{cor}}
\newcommand{\ec}{\end{cor}}

\def\R{\mathbb{R}}

\begin{document}

\tikzstyle{w}=[label=right:$\textcolor{red}{\cdots}$] 
\tikzstyle{b}=[label=right:$\cdot\,\textcolor{red}{\cdot}\,\cdot$] 
\tikzstyle{bb}=[circle,draw=black!90,fill=black!100,thick,inner sep=1pt,minimum width=3pt] 
\tikzstyle{bb2}=[circle,draw=black!90,fill=black!100,thick,inner sep=1pt,minimum width=2pt] 
\tikzstyle{b2}=[label=right:$\cdots$] 
\tikzstyle{w2}=[]
\tikzstyle{vw}=[label=above:$\textcolor{red}{\vdots}$] 
\tikzstyle{vb}=[label=above:$\vdots$] 

\tikzstyle{level 1}=[level distance=3.5cm, sibling distance=3.5cm]
\tikzstyle{level 2}=[level distance=3.5cm, sibling distance=2cm]

\tikzstyle{bag} = [text width=4em, text centered]
\tikzstyle{end} = [circle, minimum width=3pt,fill, inner sep=0pt]

\title[Flow polytopes with Catalan volumes]{Flow polytopes with Catalan volumes}
\author{Sylvie Corteel}
\address{Sylvie Corteel, IRIF, CNRS et Universit\'e Paris Diderot, 75205 Paris Cedex 13, France.
{corteel@irif.fr}}

\author{Jang Soo Kim}
\address{Jang Soo Kim, Sungkyunkwan University,
2066 Seobu-ro, Jangan-gu, 
Suwon, Gyeonggi-do 16419
South Korea. {jangsookim@skku.edu}}

\author{Karola M\'esz\'aros}
\address{Karola M\'esz\'aros, Department of Mathematics, Cornell University, Ithaca NY 14853.  \newline{ karola@math.cornell.edu}
}

\thanks{Corteel is partially supported by the project Emergences ``Combinatoire \`a Paris''. 
Kim is partially supported by
National Research Foundation of Korea (NRF) grants (NRF-2016R1D1A1A09917506) and (NRF-2016R1A5A1008055). 
M\'esz\'aros is partially supported by a National Science Foundation Grant  (DMS 1501059).}

\date{\today}

\begin{abstract} The Chan-Robbins-Yuen polytope can be thought of as the flow polytope of the complete graph with 
netflow vector $(1, 0, \ldots, 0, -1)$. The  normalized volume of the Chan-Robbins-Yuen polytope  equals the product of consecutive 
Catalan numbers, yet there is no combinatorial proof of this fact. We consider a natural generalization of this polytope, namely,  the flow 
polytope of the complete graph with netflow vector $(1,1, 0, \ldots, 0, -2)$. We show that the volume of this polytope is a certain power of $2$ times  
the product of consecutive Catalan numbers. Our proof uses constant term identities and further deepens the combinatorial mystery of why these numbers appear. 
In addition we introduce two more  families of flow polytopes whose volumes are given by product formulas. \end{abstract}

\maketitle
\section{Introduction}
\label{sec:intro}

We underscore  the wealth of flow polytopes with product formulas for volumes. The natural question arising from our study and previous work \cite{Z, Z1, CR, CRY,mm, BV2, m-prod, tesler} is: is there a unified (combinatorial?) explanation for these beautiful product formulas? All current results relating to these volumes  show these formulas as a result of various computations that surprisingly yield products. Our hope is that by identifying   three more distinguished families of flow polytopes with  beautiful product formulas for their volumes we inch closer to uncovering an illuminating explanation for these formulas. 

The {\bf flow polytope} $\F_G(\a)$ is associated to a graph $G$ on the vertex set $\{1,\ldots, n\}$ with edges directed from smaller to larger vertices and netflow vector ${\a}=(a_1, \ldots, a_n) \in \Z^n$. The  points of $\F_G(\a)$ are nonnegative flows  on the edges of $G$  so that  flow is conserved at each vertex; see Figure \ref{poly} (Section \ref{sec:def} has precise definition). Flow polytopes are closely related to Kostant partition functions \cite{BV2, mm}, Grothendieck polynomials \cite{pipe, pipe1, toric}, 
and the space of diagonal harmonics \cite{tesler, tesler1},  among others. 

\begin{figure}
\begin{center}
\includegraphics[scale=.6]{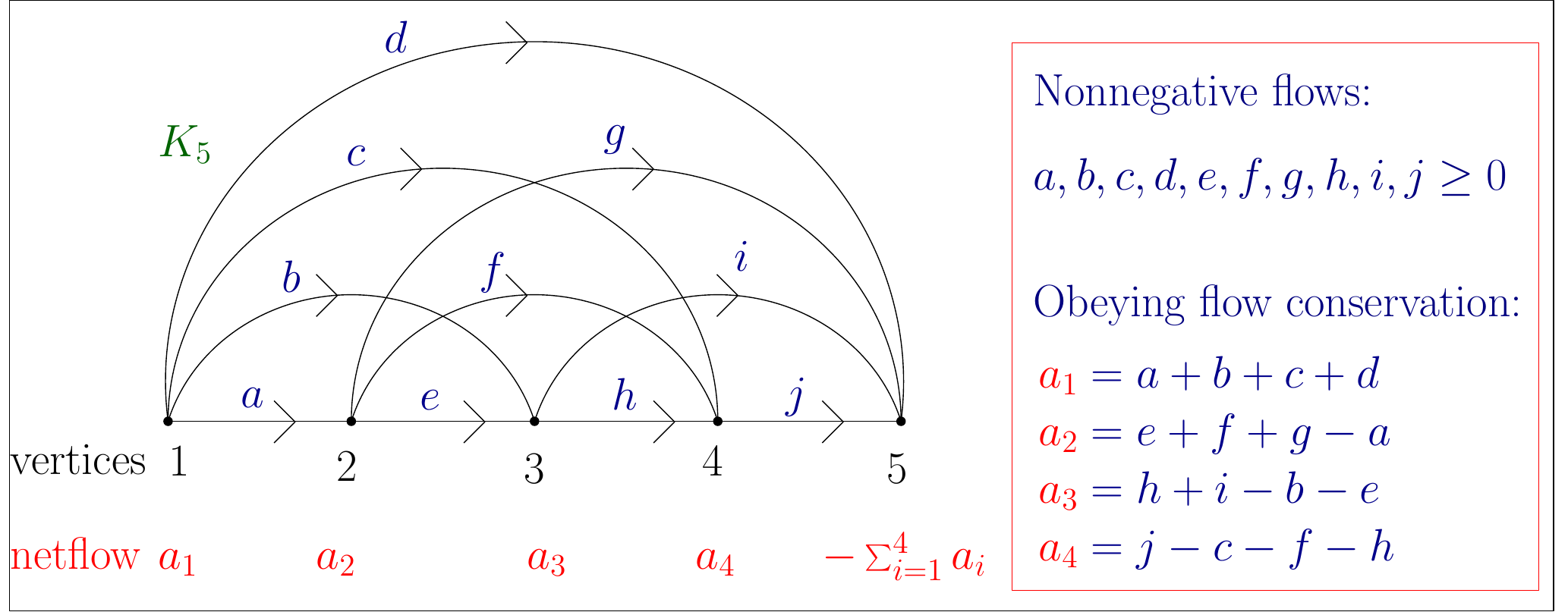}
\caption{The flow polytope $\mathcal{F}_{K_5}(a_1, a_2, a_3, a_4, -\sum_{i=1}^4 a_i)$ consists of all points $(a,b,c,d,e,f,g,h,i,j) \in \mathbb{R}^{10}$ satisfying the inequalities and equations displayed to the right of ~$K_5$.}
\label{poly}
\end{center}
\end{figure}

Perhaps the most famous flow polytope is $\F_{K_{n+1}}(1,0, \ldots, 0, -1)$, the flow polytope of the complete graph, also referred to as  the \textbf{Chan--Robbins--Yuen  polytope} ($CRY_n$) \cite{CRY}. Chan, Robbins and Yuen defined $CRY_n$ as the convex hull of the set of $n\times n$ permutation matrices $\pi$ with $\pi_{ij}=0$ if $j\geq i+2$, which can be shown to be integrally equivalent to   $\F_{K_{n+1}}(1,0, \ldots, 0, -1)$. (Thus, $CRY_n$ and $\F_{K_{n+1}}(1,0, \ldots, 0, -1)$ are combinatorially equivalent, and have the same volume and Ehrhart polynomial.) The polytope $CRY_n$ is  a face of the Birkhoff polytope, the polytope of all doubly stochastic matrices, prominent in combinatorial optimization. Remarkably, the volume of the $CRY_n$ polytope is the product of the first $n-2$ Catalan numbers, as conjectured by Chan, Robbins and Yuen in \cite{CRY} and proved by Zeilberger analytically in \cite{Z}.  Under volume in this paper we mean the normalized volume of a polytope. The  {\bf normalized volume} of a
$d$-dimensional polytope $P \subset \mathbb{R}^n$, denoted by $\vol \hspace{.02in} P$, is the volume form which 
 assigns a volume of one to the smallest $d$-dimensional integer simplex in the affine span of $P$.

 Several generalizations of  $CRY_n$   are introduced and studied in    \cite{m-prod,mm, tesler}. The volume formulas of the aforementioned polytopes are akin that of  $CRY_n$. In this paper we identify three new families of flow polytopes generalizing  $CRY_n$. In particular,  we study  the flow polytope of the complete graph with netflow vector $(1,1, 0, \ldots, 0, -2)$ and  show that its volume is a power of $2$ times  the product of consecutive Catalan numbers. Furthermore, if we take the complete graph with various multiple edges and consider the corresponding flow polytope with netflow vectors  $(1, 0, \ldots, 0, -1)$ or $(1,\ldots, 1, -n)$ we still obtain product formulas for their volumes, as a result of the generalized Lidskii formulas \cite{BV2} and the Morris (and the like) constant term identity \cite{WM}. Combinatorial proofs remain elusive, but all the more enticing.

Now we state our results regarding the three new families of polytopes we study in this paper. For definitions and background see Section \ref{sec:def}.

\begin{thm} \label{thm:2}The normalized volume of the flow polytope $\F_{K_{n+1}}(1,1,0, \ldots, 0, -2)$ is 

\begin{equation*} \vol \F_{K_{n+1}}(1,1,0, \ldots, 0, -2)= 2^{{n \choose 2}-1}\prod_{i=1}^{n-2} \Cat(i), \end{equation*} where $\Cat(i)=\frac{1}{i+1}{2i \choose i}$ is the $i$th Catalan number. 
\end{thm}

Let $\Gamma(\cdot)$ denote the Gamma function. In particular, $\Gamma(j)=(j-1)!$ when $j\in \mathbb{N}$.
\begin{thm} \label{thm:abm} Denote by $K_{n+1}^{a,b,m}$ the graph on the vertex set $[n+1]$ with each edge $(1,i)$, $i \in [2,n]$, appearing $a$ times, 
edge $(i,n+1)$, $i \in [2,n]$, appearing $b$ times,  and $(i,j)$, $1<i<j<n+1$ appearing $m$ times. Then we have that 
\begin{equation*} \vol \F_{K_{n+1}^{a,b,m}}(1,0, \ldots, 0, -1)= \frac{1}{(n-1)!} \prod_{j=0}^{n-2} \frac{\Gamma(a-1+b+(n-2+j)\frac{m}{2}) \Gamma( \frac{m}{2} )}
{\Gamma(a+j\frac{m}{2})\Gamma(b+j\frac{m}{2}) \Gamma(\frac{m}{2}+j\frac{m}{2})}. \end{equation*}
\end{thm}

\begin{thm} \label{111}  Denote by $K_{n+1}^{a,b}$ the graph on the vertex set $[n+1]$ with edges $(i,j)$,  
$1\leq i<j\leq n$, appearing with multiplicity $a$ and the edges $(i,n+1)$, $ i \in [n]$, appearing with multiplicity $b$. 
 For $n\geq 2$ and nonnegative integers
$a,b$ we have that
 \[
\vol\F_{K_{n+1}^{a,b}}(1,1, \ldots, 1, -n)=\big((b-1)n+a{\textstyle\binom{n}{2}}\big)!
\prod_{i=0}^{n-1} \frac{\Gamma(1+a/2)}{\Gamma(1+(i+1)a/2)\Gamma(b+ia/2)}.
\]

\end{thm}

The  polytope $\F_{K_{n+1}}(1, 0, \ldots, 0, -1)$ (integrally equivalent to $CRY_n$) belongs to the polytope family in Theorem \ref{thm:abm}. Indeed,   Zeilberger \cite{Z} proved the $CRY_n$  volume 
formula by specializing Morris identity (stated in Lemma \ref{morrisID}),  while Theorem \ref{thm:abm} uses the whole strength of Morris identity.  
Similarly, we make use of a Morris-type identity proved in \cite{tesler} to prove Theorem \ref{111}. It is Theorem \ref{thm:2} that makes us work significantly:  
neither the Morris, nor the Morris-type identities mentioned above work, rather we   prove a new constant term identity to tackle it.

\medskip

The outline of the paper is as follows. In Section \ref{sec:def} we give  the necessary definitions on flow polytopes. In Section \ref{sec:2} we prove Theorem \ref{thm:2}. In Section \ref{sec:abm} we prove Theorems \ref{thm:abm} and in Section \ref{sec:ab} we prove Theorem \ref{111}.  Finally, in Section \ref{sec:struc} we enumerate the vertices of the polytopes $\F_{K_{n+1}}(1,1,0, \ldots, 0, -2)$ appearing in Theorem \ref{thm:2}.

\section{Flow polytopes $\F_G(\aa)$ and Kostant partition functions}
 \label{sec:def}

The exposition of this section follows that of \cite{mm}; see \cite{mm} for more details. 

Let $G$ be a (loopless) graph on the vertex set $[n+1]$ with $N$ edges. To each edge $(i, j)$, $i< j$,  of $G$,  associate the positive
type $A_{n}$ root $\vv(i,j)=\ee_i-\ee_j$,
where $\ee_i$ is the $i$th standard basis vector in
$\mathbb{R}^{n+1}$.   Let $S_G := \{\{\vvv_1, \ldots, \vvv_N\}\}$ be the
multiset of roots corresponding to the multiset of edges of $G$. Let  $M_G$ be the
$(n+1)\times N$ matrix whose columns are the vectors in $S_G$.  Fix an
integer  vector $\g=(a_1, \ldots, a_{n+1}) \in \Z^{n+1}$ which we
call the {\bf netflow} and for which we require that $a_{n+1}=-\sum_{i=1}^n a_i$.  An {\bf $\g$-flow} $\f_G$ on $G$ is a 
vector $\f_G=(b_k)_{k \in [N]}$, $b_k \in \R_{\geq 0}$  such that $M_G \f_G=\g$. That is, for all $1\leq i \leq n+1$, we have 

 \begin{equation} \label{eqn:flowA}
\sum_{e=(g<i) \in E(G)} b(e) + a_i = \sum_{e=(i<j) \in E(G)}
b(e)  \end{equation}

Define the {\bf flow polytope} $\F_G(\g)$ associated to a  graph $G$ on the vertex set $[n+1]$ and the integer vector $\aa=(a_1, \ldots, a_{n+1})$ as the set of all $\g$-flows $\f_G$ on $G$, i.e., $\F_G=\{\f_G \in \R^N_{\geq 0} \mid M_G \f_G = \g\}$. The flow polytope 
  $\F_G(\g)$ then naturally lives in $\mathbb{R}^{N}$, where $N$ is the number  of edges of $G$. The vertices of the flow polytope 
$\F_G(\g)$ are the $\g$-flows whose supports are acyclic subgraphs of $G$ \cite[Lemma 2.1]{hille}.

Recall that the {\bf Kostant partition function}  $K_G$ evaluated at the vector ${\bf b} \in \Z^{n+1}$ is defined as

\begin{equation} \label{kost} K_G({\bf b})= \#  \Big\{ (c_{k})_{k \in [N]}
  \Bigm\vert  \sum_{k \in [N]} c_{k}  \vvv_k ={\bf b} \textrm{ and } c_{k} \in \Z_{\geq 0} \Big\},\end{equation}

\noindent where $[N]=\{1,2,\ldots,N\}$. 

The generating series of the Kostant partition function is
\begin{equation}\label{Kgs}
\sum_{{\bf b} \in \mathbb{Z}^{n+1}} K_G({\bf b}){\bf x}^{\bf b} = {\prod_{(i,j) \in E(G)}} (1-x_ix_j^{-1})^{-1} ,
\end{equation}
where ${\bf x}^{\bf b}=x_1^{b_1}x_2^{b_2}\cdots x_{n+1}^{b_{n+1}}$. In particular, 
\begin{equation} \label{gs:kostant}
K_{K_{n+1}}({\bf b})= [{\bf x}^{\bf b}]
\prod_{1\leq i<j \leq n+1} (1-x_ix_j^{-1})^{-1}.
\end{equation}

Assume that ${\bf a}=(a_1,a_2,\ldots,a_n)$ satisfies $a_i \geq  0$ for
$i=1,\ldots,n$. Let $\aa'=(a_1,a_2,\ldots,a_n, -\sum_{i=1}^n a_i)$. The {\bf generalized Lidskii formulas} of Baldoni and Vergne  state that for a graph $G$ on the vertex set $[n+1]$ with $N$ edges we have
\begin{thm}\cite[Theorem 38]{BV2}
\begin{equation} \label{eq:vol}
\vol \F_G({\bf a'}) = \sum_{{\bf i}}
\binom{N-n}{i_1,i_2,\ldots,i_n} a_1^{i_1}\cdots
a_n^{i_n}\cdot K_{G'}(i_1-t_1^G, i_2-t_2^G,\ldots, i_n - t_n^G),
\end{equation}
and 
\begin{equation} \label{eq:kost}
K_{G}({\bf a'}) = \sum_{{{\bf i}}}
\binom{a_1+t_1^G}{i_1}\binom{a_2+t_2^G}{i_2} \cdots
\binom{a_{n}+t_n^G}{i_{n}} \cdot  K_{G'}(i_1-t_1^G, i_2-t_2^G,\ldots, i_n - t_n^G),
\end{equation}
where both sums are over weak compositions ${\bf
  i}=(i_1,i_2,\ldots,i_n)$ of $N-n$ with $n$ parts which we
denote as ${\bf i} \models N-n$,  $\ell({\bf i})=n$. The graph $G'$ is the restriction of $G$ to the vertex set $[n]$. The notation  $t_i^G$, $i \in [n]$, stands for the outdegree of vertex $i$ in $G$ minus $1$. 
\end{thm}

The next three sections utilize the generalized Lidskii formulas. 

\section{A new Catalan polytope}
\label{sec:2}

In this section we prove Theorem \ref{thm:2}. Our methods rely on  \eqref{eq:vol}  and  constant term identities. 

For a Laurent series $f(x)$ in $x$, we denote the constant term by $CT_x f(x)$. We will also use the notation
\[
\CT_{x_n,\ldots,x_1} = \CT_{x_n}\cdots CT_{x_1}.
\]

We refer to the polytope of Theorem \ref{thm:2} as the ``Catalan polytope", since its volume involves Catalan numbers.  Our proof rests on the following two lemmas, whose proofs we provide after: 
\begin{lem} \label{lem1}
\[
\vol(\F_{K_{n+1}}(1,1,0,\dots,0,-2))
=\CT_{x_n,\ldots ,x_1} \frac{(x_n+x_{n-1})^{n\choose 2}}
{\prod_{1\le i<j\le n}(x_j-x_i)}.
\]  
\end{lem}

\begin{lem} \label{lem2}
\[
\CT_{x_n,\ldots ,x_1} \frac{(x_n+x_{n-1})^{n\choose 2}}
{\prod_{1\le i<j\le n}(x_j-x_i)}=2^{\binom n2 -1} \prod_{k=1}^{n-2} \Cat(k).
\]  
\end{lem}

Recall Theorem \ref{thm:2}:

\begin{theorem1*} {\it The normalized volume of the flow polytope $\F_{K_{n+1}}(1,1,0, \ldots, 0, -2)$ is }

  \begin{equation*} \vol \F_{K_{n+1}}(1,1,0, \ldots, 0, -2)= 2^{{n
        \choose 2}-1}\prod_{k=1}^{n-2} \Cat(k).\end{equation*} 
 \end{theorem1*}

\proof  Immediate from Lemmas \ref{lem1} and \ref{lem2}.
\qed

\subsection{Proving Lemma \ref{lem1}.}

Wenow show how to express the volume as a constant term identity.

\noindent \textit{Proof of Lemma \ref{lem1}.} First note that \begin{equation} \label{prop:revflow}
K_{K_{n+1}}(a_1,a_2,\ldots,a_n,-\sum_{i=1}^n a_i) = K_{K_{n+1}}(\sum_{i=1}^n
a_i, -a_n,\ldots,-a_2,-a_1). \end{equation}

By \eqref{eq:vol} and  \eqref{prop:revflow} we have that
\begin{align*}
& \vol\F_{K_{n+1}}(1,1,0, \ldots, 0, -2)\\  &= \sum_{{\bf i} \models \binom{n}{2},
  {\bf i}=(i_1,i_2,0,\ldots ,0)} \binom{\binom{n}{2}}{i_1,i_2}
\cdot K_{K_{n}}(i_1-n+1, i_2-n+2,-n+3, -n+4, \ldots,0)\\
&= \sum_{i_1+i_2={n\choose 2}} \binom{\binom{n}{2}}{i_1,i_2}
\cdot K_{K_{n}}(0,1,2,\ldots,n-4, n-3, n-2-i_2, n-1-i_{1}).
\end{align*}

We use \eqref{gs:kostant} to rewrite this as 
\[
\vol\F_{K_{n+1}}(1,1,0, \ldots, 0, -2) 
= \sum_{i_1+i_2=\binom n2} 
\binom{\binom{n}{2}}{i_1,i_2}  [{\bf
  x}^{\delta_n}x_{n-1}^{-i_2}x_n^{-i_1}] \prod_{1\leq i<j\leq n}
(1-x_ix_j^{-1})^{-1},
\]
where $\delta_n = (0,1,2,\ldots,n-1)$. Since $[{\bf x}^{\bf a}]
f = CT_{x_n,\ldots ,x_1} {\bf x}^{-{\bf a}} f$, we have
\begin{multline*}
\vol\F_{K_{n+1}}(1,1,0, \ldots, 0, -2) \\
= CT_{x_n,\ldots ,x_1} \sum_{i_1+i_2=\binom n2} {\bf x}^{-\delta_n}x_{n-1}^{i_2}x_n^{i_1}
\binom{\binom{n}{2}}{i_1,i_2} 
\prod_{1\leq i<j\leq n}
(1-x_ix_j^{-1})^{-1}.
\end{multline*}
Using $\displaystyle \prod_{1\leq i<j\leq n}
(1-x_ix_j^{-1})^{-1} = {\bf x}^{\delta_n} \prod_{1\leq i<j\leq n}
(x_j-x_i)^{-1}$ we get
\[
\vol\F_{K_{n+1}}(1,1,0, \ldots, 0, -2) = CT_{x_n,\ldots ,x_1}  \prod_{1\leq i<j\leq n}
(x_j-x_i)^{-1} \sum_{i_1+i_2=\binom n2} \binom{\binom{n}{2}}{i_1,i_2} x_{n-1}^{i_2}x_n^{i_1}.
\]
An application of the binomial theorem yields the desired result.
\qed

\subsection{Proof of Lemma~\ref{lem2}}

We need a few results before the proof  Lemma~\ref{lem2}.
The following identity was used in  \cite{Z} to prove
the volume formula for $CRY_n$:

\begin{equation}
  \label{eq:cry}
\CT_{x_{n-2},\ldots ,x_1}\prod_{j=1}^{n-2} (1-x_j)^{-2}
\prod_{1\le j<k\le  n-2} (x_k-x_j)^{-1} 
=\prod_{k=1}^{n-2} \Cat(k).
\end{equation}

Equation \eqref{eq:cry} is a special case of the Morris identity stated in Lemma \ref{morrisID}. 
We relate the constant term in Lemma~\ref{lem2} to that in 
\eqref{eq:cry}. To this end we give a combinatorial meaning to the
constant terms using matrices.

Let $\mat nm$ denote the set of $n\times m$ matrices with nonnegative
integer entries. We say that $A\in \mat nm$ is \textbf{upper triangular}
if $A_{i,j} = 0$ whenever $i>j$.  We denote by $\umat nm$ the set of
upper triangular matrices $A\in\mat nm$ with diagonal entries given by
$A_{i,i}=i-1$ for $i=1,2,\dots,\min(n,m)$.

 For $A\in \mat nm$ and an integer $k\ge1$, we define
\textbf{the $k$th row sum}
\[
r_k(A) = \sum_{i=1}^m A_{k,i}
\]
and \textbf{the $k$th hook sum}
\[
h_k(A) = \sum_{i=k+1}^m A_{k,i} - \sum_{j=1}^k A_{j,k}.
\]

For example if $m=n=4$, let $A$ be the matrix
\[
A=\left( \begin{matrix}
4 & 2 & 5 & 7\\
0 & 1 & 2 & 3\\
0 & 0 & 1 & 8\\
0 & 0 & 0 & 3
\end{matrix}\right).
\]
This gives $r_2(A)=0+1+2+3=6$, $h_2(A)=2+3-1-2=3$, $r_3(A)=9$ and $h_3(A)=0$.

For two variables $x_i$ and $x_j$ with $i<j$, we regard
$1/(x_j-x_i)$ as the Laurent series in $x_i$ and $x_j$ given by
\[
\frac{1}{x_j-x_i} = \frac{1}{x_j(1-x_i/x_j)}
=x_j^{-1}\sum_{k=0}^\infty x_i^kx_j^{-k}. 
\]

\begin{lem}\label{lem:expand}
For nonnegative integers $b$ and $m$, we have
\[
\prod_{i=1}^n (1-x_i)^{-b}\prod_{1\le i<j\le n} (x_j-x_i)^{-m}
= \sum_{\substack{A^{(1)},\dots,A^{(m)}\in\umat nn\\B \in \mat nm}}
\prod_{i=1}^n x_i ^{\sum_{j=1}^m h_i(A^{(i)}) + r_i(B)}.
\]
In particular, when $m=1$, we have
\[
\prod_{i=1}^n (1-x_i)^{-b}\prod_{1\le i<j\le n} (x_j-x_i)^{-1}
= \sum_{A\in\umat n{(n+m)}} x_1^{h_1(A)}  \dots x_n^{h_n(A)}.
\]
\end{lem}
\begin{proof}
This follows immediately from the expansions
\[
\prod_{i=1}^n (1-x_i)^{-b}
= \sum_{A\in\mat nb} x_1^{r_1(A)}  \dots x_n^{r_n(A)},
\]
\[
\prod_{1\le i<j\le n} \frac{1}{x_j-x_i}
= \sum_{A\in\umat nn} x_1^{h_1(A)}  \dots x_n^{h_n(A)}.
\]
\end{proof}

The following is the main lemma in this subsection. 

\begin{lem} \label{lem:gen}
  Suppose that $n$ is a nonnegative integer and $a, a_1,\dots,a_n$ are
  any integers with $a_1+a_2+\dots+a_n=a$.
Then 
  \begin{multline*}
CT_{x_n,\ldots,x_1} (x_{n-1}+x_n)^{\binom n2 - a}
(x_{n-1}^{a_{n-1}}x_n^{a_n} +x_{n-1}^{a_{n}}x_n^{a_{n-1}})
\prod_{i=1}^{n-2} x_i^{a_i}\prod_{1\le i<j\le n} (x_j-x_i)^{-1} \\
=2^{\binom n2-a}CT_{x_{n-2},\ldots,x_1} 
\prod_{i=1}^{n-2} x_i^{a_i}(1-x_i)^{-2} \prod_{1\le i<j\le n-2} (x_j-x_i)^{-1}.
  \end{multline*}
\end{lem}
\begin{proof}
Let $L$ be the left hand
side. Then
\begin{multline*}
L = CT_{x_n}CT_{x_{n-1}}(x_{n-1}+x_n)^{\binom n2 - a}
(x_{n-1}^{a_{n-1}}x_n^{a_n} +x_{n-1}^{a_{n}}x_n^{a_{n-1}}) \\
\times CT_{x_{n-2},\ldots ,x_1} 
\prod_{i=1}^{n-2} x_i^{a_i}\prod_{1\le i<j\le n} (x_j-x_i)^{-1} .
\end{multline*}
By Lemma~\ref{lem:expand},
\[
CT_{x_{n-2},\ldots ,x_1} 
\prod_{i=1}^{n-2} x_i^{a_i}\prod_{1\le i<j\le n} (x_j-x_i)^{-1} 
= \sum_{A\in T} x_{n-1}^{h_{n-1}(A)}  x_{n}^{h_{n}(A)},
\]
where 
\[
T = \{A\in\umat nn: h_i(A) = -a_i \mbox{ for } i=1,2,\dots,n-2\}.
\]
  Thus
  \begin{multline*}
L = CT_{x_n}CT_{x_{n-1}} \sum_{t=0}^{\binom n2 - a}
\binom{\binom n2 - a}{t} x_{n-1}^t x_n^{\binom n2 -a -t}
(x_{n-1}^{a_{n-1}}x_n^{a_n} +x_{n-1}^{a_{n}}x_n^{a_{n-1}})\\
\times \sum_{A\in T} x_{n-1}^{h_{n-1}(A)}  x_{n}^{h_{n}(A)},
  \end{multline*}
and we get
  \begin{equation}
\label{eq:L}
L= \sum_{t=0}^{\binom n2 -a}
\binom{\binom n2-a}{t} (|X_t|+|X_t'|), 
  \end{equation}
  where $X_t$ (respectively $X_t'$) is the set of matrices
  $A\in\umat nn$ such that $h_i(A) = -a_i$ for $i=1,2,\dots,n-2$, and
  $t+a_{n-1}+h_{n-1}(A)=0$ and $\binom n2 -a -t + a_n+h_n(A) =0$
  (respectively $t+a_{n}+h_{n-1}(A)=0$ and
  $\binom n2 -a -t + a_{n-1}+h_n(A) =0$). Since every matrix
  $A\in\umat nn$ satisfies $h_1(A) + \cdots + h_n(A) = -\binom n2$, we
  can omit the condition on $h_n(A)$. Therefore we can rewrite $X_t$
  and $X_t'$ as
  \begin{align*}
X_t  &= \{A\in\umat nn: h_i(A) = -a_i \mbox{ for } i=1,2,\dots,n-2, 
h_{n-1}(A) = -a_{n-1}-t\},\\
X_t'  &= \{A\in\umat nn: h_i(A) = -a_i \mbox{ for } i=1,2,\dots,n-2, 
h_{n-1}(A) = -a_{n}-t\}.
  \end{align*}

Putting 
\[
X = \bigcup_{t=0}^{\binom n2 -a} X_t, \quad
X' = \bigcup_{t=0}^{\binom n2 -a} X_t',
\]
we can rewrite \eqref{eq:L} as follows:
\begin{equation}
  \label{eq:2L}
2L = \sum_{A\in X} \binom{\binom n2 -a}{-a_{n-1}-h_{n-1}(A)}
+\sum_{A\in X'} \binom{\binom n2 -a}{-a_{n}-h_{n-1}(A)}
\end{equation}

Let 
\[
Y = \{ B\in \umat{(n-2)}{n}: h_i(A) = -a_i \mbox{ for } i=1,2,\dots,n-2 \}.
\]
Then
\[
|Y| = CT_{x_{n-2},\dots ,x_1} 
\prod_{i=1}^{n-2} x_i^{a_i} (1-x_i)^{-2} \prod_{1\le i<j\le n-2} (x_j-x_i)^{-1}.
\]

We claim that there is a bijection
$\phi:X\uplus X' \to Y\times\{0,1,\dots,\binom n2 -a\}$ such that if
$\phi(A) = (B,t)$ for $A\in X$ then $-a_{n-1}-h_{n-1}(A) =t$ or
$-a_{n-1}-h_{n-1}(A) =\binom n2-a-t$ and if $\phi(A) = (B,t)$ for
$A\in X'$ then $-a_{n}-h_{n-1}(A) =t$ or
$-a_{n}-h_{n-1}(A) =\binom n2-a-t$.  Applying this bijection to
\eqref{eq:2L} we get
\[
2L =\sum_{(B,t)\in Y\times \{0,1,\dots,\binom n2-a\}} 
\binom{\binom n2 -a}{t},
\]
which is equal to $2^{\binom n2-a} |Y|$.  Thus it is now sufficient
to find such a bijection.

We define the map
$\phi:X\uplus X' \to Y\times\{0,1,\dots,\binom n2-a\}$ by
$\phi(A) = (B,t)$ for $A\in X$ and $\phi(A) = (B',\binom n2-a - t)$
for $A\in X'$, where $t=-a_{n-1}-h_{n-1}(A)$, $B$ is the matrix obtained from
$A$ by removing the last two rows, and $B'$ is the matrix obtained
from $B$ by exchanging the last two columns.

Let $(B,t)\in Y\times\{0,1,\dots,\binom n2-a\}$.  In order to show
that $\phi$ is a bijection, we must show that there is a unique
element $A\in X\uplus X'$ such that $\phi(A) = (B,t)$.  Let
$c_i$ be the sum of entries in the $i$th column of $B$ for
$i=n-1,n$. Then we have
\[
h_1(B) + \cdots + h_{n-2}(B) = -0-1-\cdots-(n-3) +c_{n-1}+c_n =
-\sum_{i=1}^{n-2} a_i.
\]
Thus 
\begin{equation}
  \label{eq:2}
c_{n-1}+c_n = \binom{n-2}2 -\sum_{i=1}^{n-2} a_i.   
\end{equation}
We now consider the following two cases.

{\bf Case 1:} There is a matrix $A\in X$ such that
$\phi(A) = (B,t)$. In this case,
$h_{n-1}(A) = -c_{n-1}-(n-2)+A_{n-1,n} = -a_{n-1}-t$. Thus $A$ is uniquely
determined by $A_{n-1,n} = c_{n-1}+(n-2)-a_{n-1}-t$ and such a matrix $A$
exists if and only if
\begin{equation}
  \label{eq:1}
c_{n-1}+(n-2)-a_{n-1}-t \ge 0.  
\end{equation}

{\bf Case 2:} There is a matrix $A\in X'$ such that
$\phi(A) = (B,t)$. In this case,
$h_{n-1}(A) = -c_{n}-(n-2)+A_{n-1,n} = -a_{n-1}-(\binom n2-a-t)$. Thus $A$
is uniquely determined by
$A_{n-1,n} = c_{n}+(n-2)-a_{n-1}-(\binom n2-a-t)$ and such a matrix $A$
exists if and only if
\[
c_{n}+(n-2)-a_{n-1}-\left(\binom n2-a-t\right) \ge 0.
\]
Using~\eqref{eq:2}, one can check that the above inequality is
equivalent to
\begin{equation}
  \label{eq:3}
c_{n-1}+(n-1)-a_{n-1}-t \le 0.  
\end{equation}

For any integers $n,a,t$, exactly one of \eqref{eq:1} and \eqref{eq:3}
holds. Thus there is a unique element $A\in X\uplus X'$ such
that $\phi(A) = (B,t)$. This finishes the proof. 
\end{proof}

We now have all ingredients to prove Lemma~\ref{lem2}. 

\begin{proof}[Proof of Lemma~\ref{lem2}]
  If $a_i=0$ for all $i=1,2,\dots,n$ in Lemma~\ref{lem:gen}, we have

  \begin{multline*}
CT_{x_n,\ldots,x_1} (x_{n-1}+x_n)^{\binom n2} \cdot 2 \cdot 
\prod_{i=1}^{n-2} \prod_{1\le i<j\le n} (x_j-x_i)^{-1} \\
=2^{\binom n2}CT_{x_{n-2},\ldots,x_1} 
\prod_{i=1}^{n-2} (1-x_i)^{-2} \prod_{1\le i<j\le n-2} (x_j-x_i)^{-1}.
  \end{multline*}
By \eqref{eq:cry} we obtain the desired identity.
\end{proof}

\section{Morris polytopes}
\label{sec:abm}

We refer to the polytopes of Theorem \ref{thm:abm} as the  ``Morris polytopes",  as their volume formulas are byproducts of the Morris identity.  This section is devoted to proving Theorem \ref{thm:abm}, which we achieve in a sequence of lemmas.

\begin{lem}[Morris Identity \cite{Z}] \label{morrisID}
For positive integers $n$, $a$, and $b$, and  $m$, let 

\[
C(n,a,b,m) = 
CT_{x_n,\ldots ,x_1} \prod_{i=1}^n x_i^{-a}  (1-x_i)^{-b} 
\prod_{1\le i<j\le  n}(x_j-x_i)^{-m}.
\]

Then 
\begin{equation}
  \label{eq:zeilberger}
C(n,a,b,m)= \frac{1}{n!} \prod_{j=0}^{n-1} \frac{\Gamma(a+b+(n-1+j)m/2)\Gamma(m/2)}
{\Gamma(b+jm/2)\Gamma(m/2+jm/2)\Gamma(a+jm/2+1)}.
\end{equation}
\end{lem}

Recall that  $K_{n+1}^{a,b,m}$ is the graph on the vertex set $[n+1]$ with each edge $(1,i)$, $i \in [2,n]$, appearing $a$ times, 
edge $(i,n+1)$, $i \in [2,n]$, appearing $b$ times,  and $(i,j)$, $1<i<j<n+1$ appearing $m$ times.  We apply the following unpublished result of Postnikov and Stanley to $\F_{K_{n+1}^{a,b,m}}(1,0, \ldots, 0, -1)$. We note that their theorem can be seen as a special case of a version of the generalized Lidskii formulas.

\begin{thm}\cite{BV2, mm} \label{ps}
For a graph $G$ on the vertex set $[n]$, with $d_i=(\text{indegree of }i)-1$, we have
\[\vol\left(\F_{G}(1,0, \ldots, 0, -1) \right)=K_{G}(0,d_2, \ldots, d_{n-1}, -\sum_{i=2}^{n-1} d_i).\]  
\end{thm}

\begin{lem} \label{?}
For positive integers $n$, $a$, and $b$, and  $m$, we have
\begin{multline*}
\vol \F_{K_{n+1}^{a,b,m}}(1,0, \ldots, 0, -1)\\
=\CT_{x_n}\CT_{x_{n-1}}\cdots \CT_{x_1} \prod_{i=1}^{n-1} x_i^{-a+1}  (1-x_i)^{-b} 
\prod_{1\le i<j\le  n-1}(x_j-x_i)^{-m}.
\end{multline*}
\end{lem}

\proof Denote by $K_n^{m, b}$ the restriction of $K_{n+1}^{a,b,m}$ to the vertex set $[2,n+1]$. Let 
\begin{multline*}
{\bf v}= \left(0, a-1, a-1+m, a-1+2m, \ldots,a-1+(n-2)m,\right. \\
\left. -(n-1)(a-1)-{{n-1} \choose 2}m \right)  
\end{multline*}
and 
\begin{multline*}
{\bf w}= \left(a-1, a-1+m, a-1+2m, \ldots,a-1+(n-2)m, \right.\\
\left. -(n-1)(a-1)-{{n-1} \choose 2}m \right).  
\end{multline*}
Also let $${\bf x}^{\bf w}=x_1^{a-1}x_2^{a-1+m}\cdots x_{n-1}^{a-1+(n-2)m}x_n^{-(n-1)(a-1)-{{n-1} \choose 2}m}$$ and  $${\bf \tilde{x}}^{\bf \tilde{w}}=x_1^{a-1}x_2^{a-1+m}\cdots x_{n-1}^{a-1+(n-2)m}$$  Then by  Theorem \ref{ps} we have that 
\begin{align*}
&\vol \F_{K_{n+1}^{a,b,m}}(1,0, \ldots, 0, -1)  = K_{K_{n+1}^{a,b,m}}({\bf v} )
= K_{K_{n}^{b,m}}({\bf w})\\
&=[{\bf x}^{\bf w}] \prod_{i=1}^{n-1} (1-x_i{x_n}^{-1})^{-b} \prod_{1\leq i<j\leq n-1}(1-x_i{x_j}^{-1})^{-m}\\
&=[{\bf \tilde{x}}^{\bf \tilde{w}}]  \prod_{i=1}^{n-1} (1-x_i)^{-b}x_i^{(i-1)m} \prod_{1\leq i<j\leq n-1}(x_j-x_i)^{-m}\\
&=\CTn{n-1}  \prod_{i=1}^{n-1} (1-x_i)^{-b}x_i^{-a+1} \prod_{1\leq i<j\leq n-1}(x_j-x_i)^{-m}.
\end{align*}
\qed

Theorem \ref{thm:abm} states that
\begin{equation*} \vol \F_{K_{n+1}^{a,b,m}}(1,0, \ldots, 0, -1)= 
\frac{1}{(n-1)!} \prod_{j=0}^{n-2} \frac{\Gamma(a-1+b+(n-2+j)\frac{m}{2}) \Gamma( \frac{m}{2} )}{\Gamma(a+j\frac{m}{2})\Gamma(b+j\frac{m}{2}) \Gamma(\frac{m}{2}+j\frac{m}{2})}. \end{equation*}
Its proof is immediate from Lemmas \ref{morrisID} and  \ref{?}.

\section{Generalizations of the Tesler polytope}
\label{sec:ab}
 
 In this section we study generalizations of the Tesler polytope $\F_{K_{n+1}}(1, \ldots, 1, -n)$ which was introduced and studied in \cite{tesler}. It is proved in \cite{tesler} that normalized volume of  $\F_{K_{n+1}}(1, \ldots, 1, -n)$ equals
\begin{align}
\vol\F_{K_{n+1}}(1, \ldots, 1, -n) =  \frac{\binom{n}{2}!\cdot
  2^{\binom{n}{2}}}{\prod_{i=1}^n i!} = |\SYT_{(n-1,n-2,\ldots,1)}| \cdot
\prod_{i=0}^{n-1} \Cat(i), \label{eq:volformula}
\end{align}
where $\Cat(i) = \frac{1}{i+1}\binom{2i}{i}$ is the $i\textsuperscript{th}$ Catalan number
and $|\SYT_{(n-1,n-2,\ldots,1)}|$ is the number of Standard Young Tableaux of
staircase shape $(n-1,n-2,\ldots,1)$. 

Denote by $K_{n+1}^{a,b}$ the graph on the vertex set $[n+1]$ with edges $(i,j)$,  $1\leq i<j\leq n$, appearing with multiplicity $a$ and the edges $(i,n+1)$, $ i \in [n]$, appearing with multiplicity $b$. Our objective in this section is to calculate the volumes of $\F_{K_{n+1}^{a,b}}(1,1, \ldots, 1, -n)$. The Tesler polytope is a special case when we set $a=b=1$.
 
\begin{lem} \label{lem:ab}
For $n\ge2$, and nonnegative integers $a$, and $b$, 
\begin{multline*}
\vol\F_{K_{n+1}^{a,b}}(1,1, \ldots, 1, -n) \\=\CT_{x_n,\ldots ,x_1}  (x_1+\cdots+x_n)^{{{n} \choose 2}a+n(b-1)} \prod_{i=1}^n x_i^{-b+1} \prod_{1\leq i<j\leq n}
(x_j-x_i)^{-a}. 
\end{multline*}
\end{lem}

\proof
We apply \eqref{eq:vol} to $\F_{K_{n+1}^{a,b}}(1,1, \ldots, 1, -n)$. Denote by $K_n^{a}$ the restriction of  
$K_{n+1}^{a,b}$ to the vertex set $[n]$. Note that  $K_n^{a}$ is the complete graph on the vertex set $[n]$ with each edge 
appearing with multiplicity $a$. For $K_{n+1}^{a,b}$ we have $N={{n} \choose 2}a+nb$ and $r=n$ in \eqref{eq:vol}. Moreover, 
$t_1=(n-1)a+b-1, t_2=(n-2)a+b-1, t_3=(n-3)a+b-1, \ldots, t_{n-1}=a+b-1, t_n=b-1$. By \eqref{eq:vol} we obtain
\begin{align*} 
&\vol \F_{K_{n+1}^{a,b}}(1,1, \ldots, 1, -n) \\& = \sum_{{\bf i} \models N-n, \ell({\bf i})=n}
\binom{N-n}{i_1,i_2,\ldots,i_n} K_{K_n^{a}}(i_1-t_1, i_2-t_2,\ldots, i_n - t_n), \\
&=  \sum_{{\bf i} \models N-n, \ell({\bf i})=n}
\binom{N-n}{i_1,i_2,\ldots,i_n} K_{K_n^{a}}(t_n-i_n, t_{n-1}-i_{n-1},\ldots, t_1 - i_1).
\end{align*}

We use \eqref{gs:kostant} to rewrite this as 
\[
\vol\F_{K_{n+1}^{a,b}}(1,1, \ldots, 1, -n) 
= \sum_{{\bf i} \models N-n, \ell({\bf i})=n}
\binom{N-n}{i_1,i_2,\ldots,i_n} [{\bf
  x}^{{\bf t} - {\bf i}}]  \prod_{1\leq i<j\leq n}
(1-x_ix_j^{-1})^{-a},
\]
where ${\bf t} = (t_n, \ldots,t_1)$ and ${\bf i}=(i_n,\ldots, i_1)$. Since $[{\bf x}^{\bf a}]
f = \CTn{n} {\bf x}^{-{\bf a}} f$ then
\begin{multline*}
\vol\F_{K_{n+1}^{a,b}}(1,1, \ldots, 1, -n)  \\= \CTn{n} \sum_{{\bf i} \models N-n, \ell({\bf i})=n}
\binom{N-n}{i_1,i_2,\ldots,i_n} {\bf
  x}^{{\bf i} - {\bf t}}  \prod_{1\leq i<j\leq n}
(1-x_ix_j^{-1})^{-a}.  
\end{multline*}

Note that ${\bf t}=a\delta_n+(b-1, \ldots, b-1, b-1)$,  where $\delta_n = (0,1,2\ldots,n-1)$. 

Using $\displaystyle \prod_{1\leq i<j\leq n}
(1-x_ix_j^{-1})^{-a} = {\bf x}^{a\delta_n} \prod_{1\leq i<j\leq n}
(x_j-x_i)^{-a}$ we get
\begin{multline*}
\vol\F_{K_{n+1}^{a,b}}(1,1, \ldots, 1, -n)  \\= \CTn{n} \sum_{{\bf i} \models N-n, \ell({\bf i})=n}
\binom{N-n}{i_1,i_2,\ldots,i_n} {\bf
  x}^{{\bf i} - {(b-1, \ldots, b-1)}}  \prod_{1\leq i<j\leq n}
(x_j-x_i)^{-a}. 
\end{multline*}
An application of the multinomial theorem yields the desired result.

\qed

\begin{lem} \cite[Lemma 3.5]{tesler} \label{ale} For $n\geq 2$ and nonnegative integers
$a,b$ we have that
\[
\CTn{n} \,\, (x_1+\cdots+x_n)^{(b-1)n+ a\binom{n}{2}}  \prod_{i=1}^{n} x_i^{-b+1} \prod_{1\leq
  i < j \leq n} (x_i-x_j)^{-a}= 
  \]
\[  =\big((b-1)n+a{\textstyle\binom{n}{2}}\big)!
\prod_{i=0}^{n-1} \frac{\Gamma(1+a/2)}{\Gamma(1+(i+1)a/2)\Gamma(b+ia/2)}.
\]
\end{lem}

Now we are ready to prove Theorem \ref{111}.

\begin{theorem3*} For $n\geq 2$ and nonnegative integers
$a,b$ we have that
 \[
\vol\F_{K_{n+1}^{a,b}}(1,1, \ldots, 1, -n)=\big((b-1)n+a{\textstyle\binom{n}{2}}\big)!
\prod_{i=0}^{n-1} \frac{\Gamma(1+a/2)}{\Gamma(1+(i+1)a/2)\Gamma(b+ia/2)}.
\]
\end{theorem3*}

\proof Immediate from Lemmas \ref{lem:ab} and \ref{ale}.\qed

\section{The faces of the Catalan polytope}
\label{sec:struc}

The face structure of all flow polytopes of the complete graph was studied in \cite{tesler}. Here we specialize these results in order to enumerate the vertices of  
$\F_{K_{n+1}}(1,1,0, \ldots, 0, -2)$. The first part of this section follows the exposition of \cite[Section 2]{tesler}.

Let $\mathrm{rstc}_n$ denote the shifted staircase of size $n$. We use the  matrix coordinates $\{(i, j) \,:\, 1 \leq i \leq j \leq n \}$ to describe the cells of $\mathrm{rstc}_n$.
An {\bf $\aa$-Tesler tableau} $T$ (defined in \cite{tesler}) is a $(0,1)$-filling of $\mathrm{rstc}_n$  which satisfies the following three conditions:
\begin{enumerate}
\item for $1 \leq i \leq n$, if $a_i > 0$, there is at least one $1$ in row $i$ of $T$,
\item for $1 \leq i < j \leq n$, if $T(i, j) = 1$, then there is at least one $1$ in row $j$ of $T$, and
\item for $1 \leq j \leq n$, if $a_j = 0$ and $T(i, j) = 0$ for all $1 \leq i < j$, then $T(j, k) = 0$ for all $j \leq k \leq n$.
\end{enumerate}
For example, if $n = 4$ and $\aa = (7,0,3,0)$, then three $\aa$-Tesler tableaux are shown below.  We write the entries of 
$\aa$ in a column to the left of a given $\aa$-Tesler tableau.

\begin{center}
\begin{Young}
,7 & 0  & 1 &1 & 1\cr
,0 & , &  0 & 0  & 1 \cr
,3 &  ,  &, & 1 &  1\cr
, 0 &,   & ,& ,& 1 
\end{Young}
\quad
\begin{Young}
,7&  1  &  0 &  1 & 0 \cr
,0& , & 0 & 0 & 0 \cr
 ,3& ,  &, &  0 & 1\cr
 ,0 &,   & ,& ,& 1 
\end{Young}
\quad
\begin{Young}
,7 & 1  & 1  & 1 &  0\cr
,0 & , & 1 &  1 & 0\cr
,3 & ,  &, & 1 & 0\cr
,0 & ,   & ,& ,& 0
\end{Young}
\end{center}

The {\bf dimension} $\dim(T)$ of an $\aa$-Tesler tableau $T$ is $\sum_{i = 1}^n(r_i - 1)$, where 
\begin{equation*}
r_i = \begin{cases}
\text{the number of $1$'s in row $i$ of $T$} & \text{if row $i$ of $T$ is nonzero,} \\
1 & \text{if row $i$ of $T$ is zero.}
\end{cases}
\end{equation*}
In other words, $\dim(T)$ is the number of $1$'s minus the number of
nonzero rows. From left to right, the dimensions of the tableaux shown
above are $3, 1$, and $3$.

Given two $\aa$-Tesler tableaux $T_1$ and $T_2$, we write $T_1 \leq T_2$ to mean that for all $1 \leq i \leq j \leq n$ we have
$T_1(i,j) \leq T_2(i,j)$. 
It is shown in
\cite{tesler} that the $\aa$-Tesler tableaux  partially ordered by
$\leq$ with a unique minimal element adjoint form  a poset graded by dimension of the Tesler tableaux plus one. We refer to the poset as the {\bf $\aa$-Tesler tableaux poset}.

\begin{thm} \cite{tesler}
\label{face-poset-characterization}
Let $\aa=(a_1,\dots,a_n) \in (\mathbb{Z}_{\geq 0})^n$ and
$\aa'=(a_1,\dots,a_n,-\sum_{i=1}^n a_i)$.  The face poset of
$\F_{K_{n+1}}(\aa')$ is isomorphic to the $\aa$-Tesler tableaux poset.
In particular, the vertices of $\F_{K_{n+1}}(\aa')$ are in bijection
with the $\aa$-Tesler tableaux of dimension 0.
\end{thm}

We need some definitions in order to compute the number of vertices of
$\F_{K_{n+1}}(\aa)$.

A \textbf{decreasing forest} on a subset $V\subseteq [n]$ is a rooted
forest such that if $u$ is a child of $v$, then $u<v$.  For a
decreasing forest $F$, a \textbf{root} is a vertex with no parent and a
\textbf{leaf} is a vertex with no child. For example, the decreasing
forest in Figure~\ref{fig:tree1} has roots $9,2,10$ and leaves
$1,3,8,2,5$. Note that an isolated vertex is both a root and a
leaf. Note also that every connected component of $F$ has a unique
root which is the largest vertex in that component. 

\begin{figure}
  \centering
\includegraphics{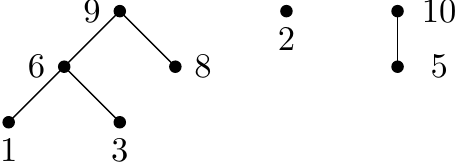}  
  \caption{A decreasing forest.}
\label{fig:tree1}
\end{figure}

We introduce another definition which is essentially the same as
decreasing forest. A \textbf{directed decreasing forest} is a directed
graph obtained from a decreasing forest by orienting each edge
$\{i,j\}$ with $i<j$ by $(i,j)$ and adding a loop $(r,r)$ for each
root $r$. Note that there is a unique way to construct a directed
decreasing forest from a decreasing forest and vice versa. For
example, the directed decreasing forest in Figure~\ref{fig:tree2}
corresponds to the decreasing forest in Figure~\ref{fig:tree1}. 

\begin{figure}
  \centering
\includegraphics{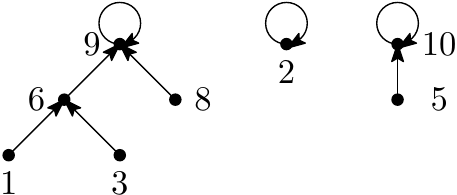}  
  \caption{A directed decreasing forest.}
\label{fig:tree2}
\end{figure}

Now we show that the number of $\aa$-Tesler tableaux of dimension $0$
is equal to the number of certain decreasing forests.

\begin{lem}\label{forest}
  Let $\aa \in (\mathbb{Z}_{\geq 0})^n$ whose nonzero entries are
  exactly in positions $s_1,s_2,\dots,s_k$. Then the number of
  $\aa$-Tesler tableaux of dimension $0$ is equal to the number of
  decreasing forests on $V$ with
  $\{s_1,\dots,s_k\}\subseteq V\subseteq [n]$ in which the leaves are
  contained in $\{s_1,s_2,\dots,s_k\}$.
\end{lem}
\begin{proof}
  It is sufficient to construct a bijection between the set
  $\mathcal T$ of $\aa$-Tesler tableaux of dimension $0$ and the set
  $\mathcal D$ of directed decreasing forest on $V$ with
  $\{s_1,\dots,s_k\}\subseteq V\subseteq [n]$ in which the leaves are
  contained in $\{s_1,s_2,\dots,s_k\}$.

  For $T\in\mathcal T$, we construct the directed graph
  $D_T=(V_T,E_T)$ as follows. The vertex set $V_T$ is the set of
  integers $i$ such that row $i$ of $T$ is nonzero. There is a
  directed edge $(i,j)\in E_T$ if and only if $T(i,j)=1$. For example,
  if $T$ is the Tesler tableau in Figure~\ref{fig:tesler}, then $D_T$
  is the directed decreasing forest in Figure~\ref{fig:tree2}. 

  \begin{figure}
    \centering
\begin{Young}
, & ,1  & ,2 & ,3 & ,4 & ,5 & ,6 & ,7 & ,8 & ,9 & ,10 \cr
,1 & 0 & 0  & 0 & 0 & 0 & 1 & 0 & 0 & 0 & 0 \cr
,2 &, & 1  & 0 & 0 & 0 & 0 & 0 & 0 & 0 & 0 \cr
,3 &, & ,  & 0 & 0 & 0 & 1 & 0 & 0 & 0 & 0 \cr
,4 &, & ,  & , & 0 & 0 & 0 & 0 & 0 & 0 & 0 \cr
,5 &, & ,  & , & , & 0 & 0 & 0 & 0 & 0 & 1 \cr
,6 &, & ,  & , & , & , & 0 & 0 & 0 & 1 & 0 \cr
,7 &, & ,  & , & , & , & , & 0 & 0 & 0 & 0 \cr
,8 &, & ,  & , & , & , & , & , & 0 & 1 & 0 \cr
,9 &, & ,  & , & , & , & , & , & , & 1 & 0 \cr
,10 &, & ,  & , & , & , & , & , & , & , & 1 \cr
\end{Young}
\caption{The Tesler tableau corresponding to the directed decreasing
  forest in Figure~\ref{fig:tree2}. Here, for readability, the row
  numbers and column numbers are indicated.}
\label{fig:tesler}
  \end{figure}

  We need to check $D_T\in \mathcal D$.  Since $\dim(T)=0$, the number
  of 1's equals the number of nonzero rows in $T$. This is equivalent
  to the condition that in $D_T$ the number of vertices equals the
  number of edges. Consider a connected component $C$ of $D_T$. Here,
  we assume that two vertices are connected if there is a path from
  one vertex to another ignoring the orientations of the edges in the
  path.  By the second condition (2) of the definition of $\aa$-Tesler
  tableau, for every vertex $i$ of $D_T$, there is an edge $(i,j)$
  with $i\le j$.  Thus, the vertex with largest label in $C$ has a
  loop.  Since $C$ is connected, if $C$ has $k$ vertices, then $C$
  must have at least $k-1$ except loops. Together with the loop at the
  largest vertex, $C$ has at least $k$ edges. If $C$ has exactly $k$
  edges, then $C$ must be a directed tree with a loop attached at the
  largest vertex. Moreover, $C$ is a directed decreasing tree for the
  following reason. If we follow a directed path, by the second
  condition (2) of the definition of $\aa$-Tesler tableau, we can
  always find a loop at the end. If $C$ is not a directed decreasing
  tree then there is a vertex of out-degree at least $2$, which
  implies that there are at least two loops. This is a contradiction
  to the fact that $C$ has $k$ edges.

  Thus we have $D_T\in \mathcal D$. It is easy to see that the map
  $T\mapsto D_T$ is a desired bijection.
\end{proof}

Using the previous lemma, we can compute the number of vertices of
$\F_{K_{n+1}}(\aa)$ when $\aa$ has two nonzero elements.

\begin{thm}
Let $n=r+s+2$ and 
\[
\aa=(1,\overbrace{0,\dots,0}^r,1,\overbrace{0,\dots,0}^s,-2).
\]
Then the number of vertices of $\F_{K_{n+1}}(\aa)$ is $2^{r+1}3^s$. 
\end{thm}
\begin{proof}
  By Theorem~\ref{face-poset-characterization} and Lemma~\ref{forest},
  the number of vertices of $\F_{K_{n+1}}(\aa)$ is equal to the number
  of decreasing forests on $V$ such that
  $\{1,r+2\}\subseteq V\subseteq[r+s+2]$ and the leaves are contained
  in $\{1,r+2\}$. Suppose that $F$ is such a decreasing forest. Since
  every tree in $F$ has at least one leaf, $F$ has at most $2$
  trees. We will count how many ways to construct $F$ in the following
  two cases.

  Case 1: $F$ has two trees $T_1$ and $T_2$, where $T_1$ has only one
  leaf $1$ and $T_2$ has only one leaf $r+2$. Since $F$ is a
  decreasing forest and each tree has only one leaf, each tree is
  determined by its vertices. For $2\le i\le r+1$, we have two
  possibilities: $i$ is a vertex of $T_1$ or not. For
  $r+3\le j\le s+r+2$, we have three possibilities: $j$ is a vertex of
  $T_1$, a vertex of $T_2$ or not a vertex of them. Thus there are
  $2^r3^s$ ways to construct such $F$. 

  Case 2: $F$ has only one tree. Then $F$ has two leaves which are $1$
  and $r+2$ or only one leaf which is $1$. Note that $r+s+2$ is the
  unique root in $F$. Let $A$ (resp.~$B$) be the set of vertices in
  the unique path from $1$ (resp.~$r+2$) to $r+s+2$. Then $F$ is
  uniquely determined by $A$ and $B$. Let $m=\min(A\cap B)$. Observe
  that $r+2\le m\le r+s+2$ and we have $m=r+2$ if and only if $F$ has
  only one leaf.  We define two sets $X$ and $Y$ as follows.
\[
X = A-\{1,m\}, \qquad Y = \{i\in B: r+2<i\le m\}.
\]
Then $X$ and $Y$ satisfy
\begin{enumerate}
\item $X\cap Y=\emptyset$,
\item $X \subseteq  \{2,3,\dots,r+1,r+3,r+4,\dots,r+s+2\}$,
\item $Y \subseteq \{r+3,r+4,\dots,r+s+2\}.$
\end{enumerate}
The two sets $A$ and $B$ can be reconstructed from $X$ and $Y$ by
\begin{align*}
A &= X\cup\{1,\max (Y\cup\{r+2\})\}, \\  
B &= (Y \cup\{r+2\}) \cup \{i\in A: i>\max (Y \cup\{r+2\})\}.
\end{align*}
Thus, $X$ and $Y$ determine $F$. Moreover, any two sets $X$ and $Y$
satisfying the above three conditions will make a decreasing forest
$F$ considered in this case. Thus the number of $F$s in this case is
equal to the number of two sets $X$ and $Y$, which is $2^r3^s$.

By the above two cases, we obtain that the theorem.
\end{proof}

As a corollary we obtain the number of vertices of our main flow
polytopes.

\begin{cor}
The number of vertices of   $\F_{K_{n+1}}(1,1,0, \ldots, 0, -2)$ is
equal to $2\cdot 3^{n-2}$. 
\end{cor}

\section*{Acknowledgements} This work started during a stay of the
second and third authors at the Universit\'e Paris 7 Diderot. The
third author is grateful for the invitation from, support of and
hospitality of the Universit\'e Paris 7.  The authors are grateful to
Alejandro Morales for sharing his Sage codes and Mich\`ele Vergne for helpful
discussions.


\begin{thebibliography}{10}

\bibitem{BV2}
W.~Baldoni and M.~Vergne.
\newblock {K}ostant partitions functions and flow polytopes.
\newblock {\em Transform. Groups}, 13(3-4):447--469, 2008.

\bibitem{CR}
C.S. Chan and D.P. Robbins.
\newblock On the volume of the polytope of doubly stochastic matrices.
\newblock {\em Experiment. Math.}, 8(3):291--300, 1999.

\bibitem{CRY}
C.S. Chan, D.P. Robbins, and D.S. Yuen.
\newblock On the volume of a certain polytope.
\newblock {\em Experiment. Math.}, 9(1):91--99, 2000.

\bibitem{pipe1}
L.~Escobar and K.~M\'esz\'aros.
\newblock Subword complexes via triangulations of root polytopes.
\newblock \arxiv{1502.03997}, 2015.

\bibitem{toric}
L.~Escobar and K.~M\'esz\'aros.
\newblock Toric matrix {S}chubert varieties and their polytopes.
\newblock {\em Proc. Amer. Math. Soc., to appear}, 2016.
\newblock \arxiv{1508.03445}.

\bibitem{hille}
L.~Hille.
\newblock Quivers, cones and polytopes.
\newblock {\em Linear Algebra Appl.}, (365):215--237.

\bibitem{tesler1}
R.I. Liu, K.~M\'esz\'aros, and A.H. Morales.
\newblock Flow polytopes and the space of diagonal harmonics.
\newblock \arxiv{1610.08370}, 2016.

\bibitem{m-prod}
K.~M\'esz\'aros.
\newblock Product formulas for volumes of flow polytopes.
\newblock {\em Proc. Amer. Math. Soc.}, (3):937--954, 2015.

\bibitem{pipe}
K.~M\'esz\'aros.
\newblock Pipe dream complexes and triangulations of root polytopes belong
  together.
\newblock {\em SIAM J. Disc. Math., to appear}, 2016.
\newblock \arxiv{1502.03991}.

\bibitem{mm}
K.~M\'esz\'aros and A.~H. Morales.
\newblock Flow polytopes of signed graphs and the {K}ostant partition function.
\newblock {\em Int. Math. Res. Notices}, (3):830--871, 2015.

\bibitem{tesler}
K.~M\'esz\'aros, A.H. Morales, and B.~Rhoades.
\newblock The polytope of {T}esler matrices.
\newblock {\em Selecta Mathematica, to appear}, 2016.
\newblock \arxiv{1409.8566v2}.

\bibitem{WM}
W.G. Morris.
\newblock {\em Constant Term Identities for Finite and Affine Root Systems:
  Conjectures and Theorems}.
\newblock PhD thesis, University of Wisconsin-Madison, 1982.

\bibitem{Z}
D.~Zeilberger.
\newblock Proof of a conjecture of {C}han, {R}obbins, and {Y}uen.
\newblock {\em Electron. Trans. Numer. Anal.}, 9:147--148, 1999.

\bibitem{Z1}
D.~Zeilberger.
\newblock Sketch of a {P}roof of an {I}ntriguing {C}onjecture of {K}arola
  {M}eszaros and {A}lejandro {M}orales {R}egarding the {V}olume of the $d_n$
  {A}nalog of the {C}han-{R}obbins-{Y}uen {P}olytope ({O}r: {T}he
  {M}orris-{S}elberg {C}onstant {T}erm {I}dentity {S}trikes {A}gain!).
\newblock {\em http://arxiv.org/abs/1407.2829}, 2014.

\end{thebibliography}
\end{document}